\documentclass[12pt,a4paper]{article}
\usepackage{graphicx,amsmath,amsfonts,amsthm}
\usepackage{pdflscape}
\usepackage[all]{xy}
\CompileMatrices
\let\nc\newcommand
\let\dmo\DeclareMathOperator
\nc\fo{\mathfrak{o}}
\nc\Q{\mathbb{Q}}
\nc\Z{\mathbb{Z}}
\nc\Ql{\Q_\ell}
\nc\Zl{\Z_\ell}
\nc\cM{\mathcal{M}}
\nc\sX{\mathcal{X}}
\nc\HM{H_\cM}
\nc\HMO{H_{\cM/\fo}}
\nc\HMf{H_{\cM,f}}
\nc\HMlf{H_{\cM,\ell-f}}
\nc\HMfp{H_{\cM,f'}}
\nc\HMnr{H_{\cM,nr}}
\nc\Fbar{\overline F}
\nc\kbar{\bar k}
\nc\cC{\mathcal{C}}
\nc\cD{\mathcal{D}}
\nc\cH{H}
\nc\La{\Ql}
\nc\cut[1]{}

\nc\Y[1]{Y_{\langle#1\rangle}}
\nc\Yb[1]{\bar Y_{\langle#1\rangle}}
\dmo{\Spec}{Spec}\dmo{\Gal}{Gal}
\dmo{\loc}{res} 
\dmo{\reg}{reg}\dmo{\gr}{gr}
\dmo{\Hom}{Hom}\dmo{\End}{End}\dmo{\im}{im}
\dmo{\Fil}{Fil}

\nc\notdivides{\not|}
\nc\kv{k_v}
\nc\G{\Gamma}
\nc\defeq{\mathop{:=}}
\nc\Ubar{\overline U}\nc\Ybar{\overline Y}
\nc\cF{\mathcal F}
\nc\AJl{AJ_\ell}
\nc\Kp[3]{K_{#1}^{\prime(#2)}#3}

\nc\xymat[2][]{
  \entrymodifiers={++!!<0pt,.8ex>}
  \vcenter{\xymatrix#1{#2}}}
\SelectTips{cm}{12}
\UseTips
\nc\nosp[1][2.5]{\hspace{-#1em}}
\nc\flr[1]{*[l]{#1\nosp}}
\newdir{ (}{!/-1ex/@^{(}}
\newdir{ )}{!/-1ex/@_{(}}
\nc{\har}{\ar@{{ (}->}}  
\nc{\hhar}{\ar@{{ )}->}} 
\nc\dar{\ar@{{}{--}>}}   
\nc\eqar{\ar@{=}}
\nc\dblar{\ar@<0.8ex>[r]\ar@<-0.8ex>[r]}
\nc\tplar{\ar@<1.5ex>[r]\ar[r]\ar@<-1.5ex>[r]}
\nc\mlb[1]{\displaystyle\boxed{#1}}

\newcommand{\mapright}[1]{%
  \stackrel{#1}{\longrightarrow}}

\newcommand{\isomarrow}{\overset{\sim}\longrightarrow}

\def\doublerightarrow{\mathrel{\vcenter{\hbox{$\rightarrow$}\kern-8pt
\hbox{$\rightarrow$}}}}
\def\doubleleftarrow{\mathrel{\vcenter{\hbox{$\leftarrow$}\kern-8pt
\hbox{$\leftarrow$}}}}
\def\multiplerightarrow{\mathrel{\vcenter{\hbox{$\rightarrow$}\kern-6pt
\hbox{ \vdots\hss}\kern-8pt\hbox{$\rightarrow$}}}}
\def\triplerightarrow{\mathrel{\vcenter{\hbox{$\rightarrow$}\kern-6pt
\hbox{$\rightarrow$}\kern-6pt\hbox{$\rightarrow$}}}}
\def\multipleleftarrow{\mathrel{\vcenter{\hbox{$\leftarrow$}\kern-6pt
\hbox{ \vdots\hss}\kern-8pt\hbox{$\leftarrow$}}}}

\theoremstyle{plain} 
\newtheorem{thm}[subsection]{Theorem}
\newtheorem{prop}[subsection]{Proposition}

\newtheorem{corol}[subsection]{Corollary}
\newtheorem{conj}[subsection]{Conjecture}
\theoremstyle{definition}

\title{Integral elements of $K$-theory and products of modular curves
  II}
\author{A J Scholl}
\date{}

\begin{document}
\maketitle
\begin{abstract}
  We discuss the relationship between different notions of
  ``integrality'' in motivic cohomology/K-theory which arise in the
  Beilinson and Bloch-Kato conjectures, and prove their equivalence in
  some cases for products of curves, as well as obtaining a general
  result, first proved by Jannsen (unpublished), reducing their
  equivalence to standard conjectures in arithmetic algebraic
  geometry. 
\end{abstract}

\section{Introduction}
\label{sec:intro}

This paper is a continuation of \cite{SI}. Its main aim is to give
an unconditional proof of the following comparison between two
different notions of integral motivic cohomology, which was (in the
special case $i=3$, $n=d=2$) stated (and used) without proof in
\cite[2.3.10]{SI}. (I am grateful to those who insisted to me that
this gap be filled.)

\begin{thm}
  \label{thm:main}
  Let $F$ be a number field, with ring of integers $\fo$.  Let
  $C_1,\dots, C_d$ be smooth projective curves over $F$, and let
  $M\subset h(\prod C_j)$ be a submotive of the Chow motive of their
  product. Let $0<i\le 2n-1$. Then if $n\ge d$, the integral motivic
  cohomology $\HMO^i(M,n)$ and the unramified motivic cohomology
  $\HMnr^i(M,n)$ coincide.
\end{thm}

(Of course, one expects this to hold for any Chow motive without the
condition $n\ge d$, and even the stronger statement in which $\HMnr$
is replaced by the Bloch-Kato $\HMf$-subgroup.)  We prove this using a
rather general compatibility in \'etale cohomology
(\ref{prop:compat2}), plus Soul\'e's bounds on $K$-groups of special
varieties over finite fields \cite{SoK}.

We first review the definitions of the various objects in Theorem
\ref{thm:main}. More generally, let $(F,\fo)$ be one of the following:
\begin{enumerate}
\item $F$ a number field, $\fo$ its ring of integers or a localisation
  of it;
\item $\fo$ a Henselian discrete valuation ring whose field of
  fractions $F$ has characteristic $0$, and whose residue field is
  finite.
\end{enumerate}
Let $U/F$ be a proper and smooth scheme. Then there are defined
motivic cohomology groups $\HM^i(U,n) = \HM^i(U,\Q(n))$, for integers
$i$, $n$. With rational coefficients, one has a $K$-theoretic
interpretation (or, if one prefers, definition):
\[
\HM^i(U,n) = K^{(n)}_{2n-i}U \subset K_{2n-i}(U)\otimes\Q
\]
the eigenspace on which Adams operators $\psi^q$ act as multiplication
by $q^n$.

If $U$ extends to a regular scheme $X$, proper and flat over $\fo$,
then the integral motivic cohomology is defined to be 
\[
\HMO^i(U,n) \defeq \im\left[ K_{2n-i}^{(n)}(X) \to \HM^i(U,n) \right]
\]
If $M$ is an effective Chow motive, then $X= e\cdot h(U)$ for some $U$
and some idempotent $e\in\End h(U)$. One may choose $U$ in such a way
that it has a regular proper model $X$, and the
subspaces
\[
\HM^i(M,n)=e\cdot\HM^i(U,n),\quad \HMO(M,n)=e\cdot\HMO^i(U,n)
\]
of $\HM^i(U,n)$ depend functorially only on $M$ (this is the main
result of \cite[\S1]{SI}). The integral motivic cohomology groups
$\HMO^*(M,*)$ feature in Beilinson's conjectures on special values of
$L$-functions \cite{Be1,Be2,RSS}.

There is defined an $\ell$-adic regulator map, with values in continuous
$\ell$-adic cohomology \cite{JC}
\[
\reg_\ell\colon \HM^i(U,n) \to H^i(U,\Ql(n)).
\]
If $i\ne 2n$ then one knows that the composite 
\[
\HM^i(U,n) \to H^i(U,\Ql(n))  \to H^i(U\otimes\Fbar,\Ql(n))
\]
is zero, so that the Hochschild-Serre spectral sequence in continuous
$\ell$-adic cohomology induces a homomorphism, the  $\ell$-adic
Abel-Jacobi map
\[
\AJl \colon \HM^i(U,n) \to H^1(F,V_\ell)).
\]
Here we have written $V_\ell = H^{i-1}(U\otimes\Fbar,\Ql(n))$ for the
$\ell$-adic cohomology of the geometric fibre. Let $v$ be a prime of
$F$ not dividing $\ell$, with residue field $k_v$, and $F_v$ the
completion of $F$ at $v$. Let $G_v=\Gal(\Fbar_v/F_v)$,
$I_v=\Gal(\Fbar_v/F_v^{nr})$ the inertia group, and
$\G_v=\Gal(\kbar_v/k_v)=G_v/I_v$.\footnote{%
  In case (ii), we mean that $v$ is the canonical place of $F$, so 
  $G_v=G$, and that the residue characteristic of $F$ is different
  from $\ell$.} %
Recall the exact sequence of ramified and unramified cohomology
\[
\label{eq:unram}
\xymat{
  0 \ar[r] & H^1(\G_v, V_\ell^{I_v}) \ar[r] \eqar[d]
  & H^1(G_v,V_\ell) \ar[r] & H^1(I_v,V_\ell)^{\G_v} \ar[r]
  \eqar[d] & 0
  \\
  & H^1_{nr}(F_v,V_\ell) & & H^1_{ram}(F_v,V_\ell)
}
\]
Let $\loc_v\colon H^1(F,V_\ell) \to H^1(F_v,V_\ell)$ be the
restriction map.
Bloch and Kato \cite{BK} define a subspace $H^1_f(F_v,V_\ell)$ which
coincides with $H^1_{nr}(F_v,V_\ell)$ if $\ell\ne p_v$, and use this
to define a subspace of motivic cohomology by
\begin{equation}
  \label{eq:hmf}
  \HMf^i(X,n) = \bigcap_{v,\ell}(\loc_v\circ \AJl)^{-1}(
  H^1_f(F_v,V_\ell)) 
\end{equation}
--- in the notation of Bloch-Kato and Fontaine--Perrin-Riou, $V_\ell$ is the realisation of the motive
$V=h^{i-1}(U)(n)$, and  they write $\HMf^1(V)$ for the
group \eqref{eq:hmf}.  Implicit in Bloch-Kato's generalisation of the
Beilinson conjectures is part (i) of the following conjecture (and see
already \cite[4.0.(b)]{BeH} for the case $\ell\ne p_v$) --- part (ii) is
folklore:

\begin{conj}
  (i) $\HMf^i(U,n)=\HMO^i(U,n)$.

  (ii) for fixed $v$ the subspace
  \[
  \ker\left[ \HM^i(U,n) \to H^1(F_v,V_\ell)/H^1_f(F_v,V_\ell) \right]
  \]
  is independent of $\ell$.
\end{conj}

Let us from now on ignore the places $v$ dividing $\ell$ (which, to be
sure, are the most interesting ones) and define
\[
\HMnr^i(U,n) = \bigcap_{v,\ell\ne p_v}(\loc_v\circ \reg_\ell)^{-1}(
  H^1_{nr}(F_v,V_\ell)) 
\]
The ring $\End h(U)$ of correspondences on $U$ (for rational
equivalence) acts on everything in sight and so for a submotive
$M\subset h(U)$ the groups $\HMf^i(M,n)^0 \subset \HMnr^i(M,n) \subset
\HM^i(M,n)$ are defined. 

It is well known that one has $\HMO\subset \HMnr$ (we recall the proof
in the next section) and even that $\HMO\subset \HMf$ under suitable
hypotheses\dots  (Similar statements hold for $\ell= char(k)$, see for
example \cite{Nek,Niz}).

Jannsen showed (unpublished) that the equality of $\HMO$ and $\HMnr$
would follow from two standard conjectures: the monodromy-weight
conjecture on the action of inertia on $\ell$-adic cohomology, and his
generalisation of the Tate conjecture on algebraic cycles to arbitrary
varieties over finite fields. See \ref{cor:uwe} below. After reviewing
some of what is known in the next section, will prove a rather general
compatibility in $\ell$-adic cohomology, from which Jannsen's result
will be a corollary.

For historical reasons I have kept to the old definition of motivic
cohomology using $K$-theory, rather than higher Chow groups. It should
not be hard to rewrite everything here in terms of higher Chow groups,
using the localisation techniques of Levine \cite{Lev}. However there
are no new phenomena to be expected when working with
$\Z$-coefficients, if only because, for a $\Zl$-representation $T$ of
$\Gal(\Fbar/F)$ (for $F$ local or global) Bloch and Kato define
$H^1_f(F,T)$ to be simply the preimage, via the natural map
$H^1(F,T)\to H^1(F,T\otimes\Ql)$ of the subspace
$H^1_f(F,T\otimes\Ql)\subset H^1(F,T\otimes\Ql)$. Moreover, the
integral groups ``without denominators'' are only meaningful in the
presence of a regular model $X$ of $U$, not just a regular alteration.

\section{Preliminaries}
\label{sec:prelim}

For completeness, let us first recall what happens when $i=2n$. In
this case, the localisation sequence of $K$-theory shows that
$\HMO^{2n}(U,n)$ and $\HM^{2n}(U,n)$ are equal; this group is
$CH^n(U)\otimes\Q$, the Chow group of codimension $n$ cycles on $U$.
In this case the cycle class map $\HM^{2n}(U,n)\to
H^{2n}(\Ubar,\Ql(n))$ is non-zero, and its kernel is
$\HM^{2n}(U,n)^0\defeq CH^n(U)^0\otimes\Q$, the subgroup of cycles
homologically equivalent to zero. The Abel-Jacobi homomorphism is a
map from $\HM^{2n}(U,n)^0$ to $H^1(F,V_\ell)$, and the obstruction to
the equality $\HMnr^{2n}(U,n)^0=\HM^{2n}(U,n)^0$ lies in the ramified
cohomology groups
\begin{equation}
  \label{eq:obs-chow}
  H^1(I_v, V_\ell)^{\G_v}
  =\Hom_{\G_v}(\Ql(1-n), H^{2n-1}(\Ubar,\Ql(n-1))_{I_v}).
\end{equation}
The monodromy-weight conjecture (recalled as \ref{conj:mwc} below)
implies that the $I_v$-coinvariants of $H^{2n-1}(\Ubar,\Ql)$ have
weights $\ge 2n-1$, and therefore that the obstruction group
\eqref{eq:obs-chow} vanishes. In other words, $\HMnr^{2n}(U,n)^0
\subset \HMO^{2n}(U,n)^0 = \HM^{2n}(U,n)^0$, with equality if the
monodromy-weight conjecture holds.

Since $\HM^i(U,n)\subset K_{2n-i}U\otimes\Q$ vanishes for $i>2n$, we
assume henceforth that $q\defeq 2n-i>0$.

For the moment suppose that we are in setting (i).  Write $\fo_{(v)}$
for the localisation of $\fo$ at $v$, $\fo_v$ for its completion, and
$\kv$ for its residue field. Assume that $U$ has a regular and proper
model $X$ over $\fo$.  Then from the localisation sequences
\begin{gather*}
  K_qX \to K_qU \to \coprod_v K'_{q-1}X\otimes\kv
  \\
  K_qX\otimes\fo_{(v)} \to K_qU \to K'_{q-1}X\otimes\kv
  \\
  K_qX\otimes\fo_v \to K_qU\otimes F_v \to K'_{q-1}X\otimes\kv
\end{gather*}
we see that 
\[
\HMO^i(U,n)=\ker\left [ \HM^i(U,n) \to \coprod_v \frac{\HM^i(U\otimes
    F_v,n)}{\HMO^i(U\otimes F_v,n) } \right ]
\]
(cf. \cite[1.3.5--6]{SI}). Since by definition the corresponding
identity holds for $\HMnr$, the comparison between $\HMO$ and $\HMnr$ is
reduced to the local case.

We also recall that both the integrality and the unramified conditions
are stable under finite extensions $F'/F$: under the inclusion 
$\HM^i(U,n)\subset \HM^i(U\otimes F',n)$ one has
\begin{align*}
  \HMO^i(U,n)&= \HM^i(U,n) \cap \HMO^i(U\otimes F',n)
  \\
  \HMnr^i(U,n)&= \HM^i(U,n) \cap \HMnr^i(U\otimes F',n)
\end{align*}
which for $\HMnr$ is clear from the definition, and for $\HMO$ follows
from \cite[\S1]{SI}.

For the rest of the paper we will assume that we are in the local
case (ii): thus $F$ is local, with valuation ring $\fo$ and finite
residue field $k$, and write $S=\Spec\fo=\{\eta,s\}$ as usual.
Let $f\colon X\to S$ be proper and flat, with special fibre
$g\colon Y=X_s \to \Spec k$ and generic fibre $U=X\setminus Y=X_\eta$.
Let $d=\dim U$, and write $G=\Gal(\Fbar/F)$, $I$ for the inertia
subgroup of $G$ and $\G=\Gal(\kbar/k)=G/I$.

We consider the analogue of $\AJl$ on $X$ itself. By the proper
base-change theorem 
\[
H^0(S,R^i f_*\Ql(n))
= H^0(s,R^i g_*\Ql(n)) = H^i(\Ybar,\Ql(n))^\G=0
\]
since by Deligne \cite{DeW2}, the weights of $H^i(\Ybar,\Ql(n)$ are $\le (i-2n)$,
hence nonzero. So from the Hochschild-Serre spectral sequence we
obtain an edge homomorphism
\[
e_1\colon H^i(X,\Ql(n)) \to  
H^1(S,\cF) \qquad\text{where }\cF = R^{i-1} f_*\Ql(n).
\]
Composing with the Chern character $ch\colon K_qX \to H^i(X,\Ql(n))$,
we obtain a commutative diagram, in which the bottom row is exact:
\begin{equation}
  \label{eq:diag1}
  \xymat{
    & K_qX \ar[d]_{e_1\circ ch} \ar[r] & K_qU \ar[d]^{\AJl} \\
    & H^1(S,\cF) \ar[d] \ar[r] & H^1(\eta,\cF_\eta) \eqar[ddd] \\
    & H^1(S,j_*\cF_\eta) \eqar[d] \\
    & H^1(s,i^*j_*\cF_\eta) \eqar[d] \\
    0 \ar[r] & H^1(\G,\cF_{\bar\eta}^I) \ar[r] & H^1(G,\cF_{\bar\eta})
    \ar[r] & H^1(I,\cF_{\bar\eta})^\G
  }
\end{equation}
This shows that $\HMO^i(X,n)\subset
\HMlf^i(X,n)$ whenever $\ell\ne char(k)$, as mentioned in the
introduction. 

We next review when the obstruction group $H^1(I,\cF_{\bar\eta})^\G$ can be
non-zero. First recall:

\begin{conj}[Monodromy-weight conjecture]
  \label{conj:mwc}
  Let $W_\bullet$ denote the weight filtration on $H^j(\Ubar,\Ql)$,
  and let $N\colon H^j(\Ubar,\Ql) \to H^j(\Ubar,\Ql)(-1)$ denote the
  ``logarithm of monodromy'' operator. Then for each $r\ge0$,
  $N^r$ induces an isomorphism
  \[
  \bar N^r \colon gr^W_{j+r}H^j(\Ubar,\Ql) \isomarrow
  gr^W_{j-r}H^j(\Ubar,\Ql)(-r).
  \]
\end{conj}

Assume that $X$ is regular, and that $Y$ is a reduced strict normal
crossings divisor in $X$. Then the weight spectral sequence of
Rapoport-Zink \cite[]{RZ} controls the weights of $H^j(\Ubar,\Ql)$;
let $h=h(X)$ be the least positive integer such that no set of more
than $h$ components of $Y$ has non-trivial intersection. Then
\[
\gr^W_w H^j(\Ubar,\Ql) \ne 0\ \Rightarrow\ 
\max\{0,j-h, 2d-j\} \le w \le \min\{2j, j+h, 2d\}.
\]
In general we may replace $U$ by an alteration $U'$ for which such a
model $U'\subset X'$ exists, and take $h=h(X')$.

Therefore if $H^1(I,\cF_{\bar\eta})^\G =
\Hom_{\G}(\Ql(1-n),H^{i-1}(\Ubar,\Ql)_I)$ is non-zero, the pair
$(i,n)$ must satisfy the inequalities
\[
n\le d+1,\quad n \le i \le 2d \quad \text{and}\quad i\ge 2n-h-1.
\]
We also have the obvious inequality $i\le 2n$.  So far we have not
used the monodromy-weight conjecture; if we assume it, then the
weights of $H^j(\Ubar,\Ql)_I$ are all $\ge j$, whence we have an
additional inequality $i\le 2n-1$, which just excludes the case $i=2n$
already considered at the beginning of this section.

For $U$ a product of curves, Theorem \ref{thm:main} therefore shows
that:
\begin{itemize}
\item  in the region $n>d+1$, one has $\HMO^i(U,n)=\HM^i(U,n)$ (for this
  the compatibility \ref{prop:compat3} is not needed, only the
  computations on the special fibre at the end of this section); and
\item along the lines $n=d$ and $n=d+1$ the integrality conditions
  (which are in general non-trivial) coincide.
\end{itemize}
Notice also that over a number field one expects $\HM^i(U,n)=0$ as
soon as $i>2d+1$.

\begin{center}
\includegraphics[scale=0.3,viewport=-50 0 1100 650,clip]{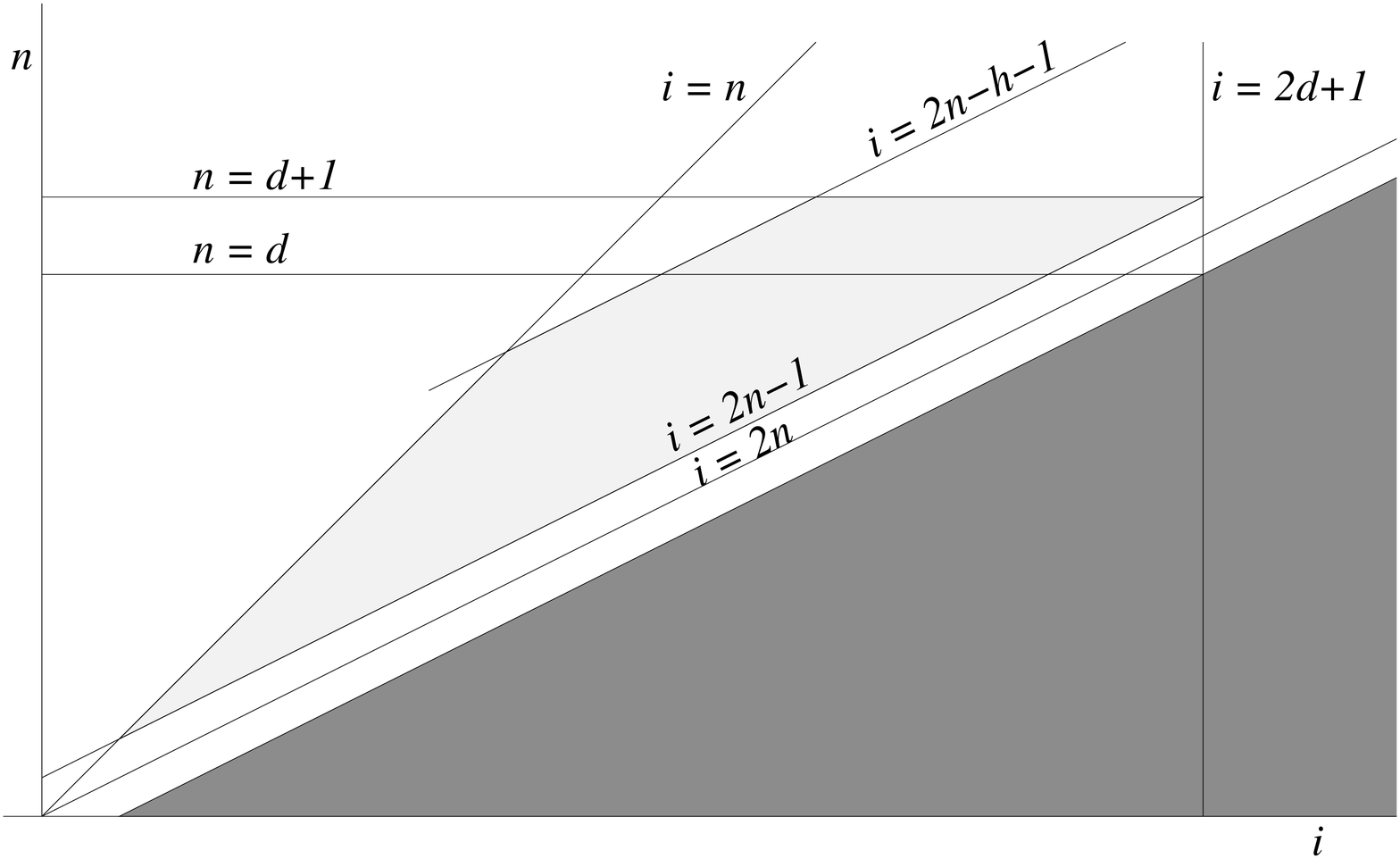}
\\
\mbox{ }
\end{center}

To go further we want to enlarge the diagram \eqref{eq:diag1} to
\begin{equation}
  \label{eq:diag2}
  \xymat{
    & K_qX \ar[r] \ar[d] & K_qU \ar[r]^{\alpha} \ar[d] \ar[dr]^{\beta}
    & K'_{q-1}Y
    \dar[d]^\phi \\
    0 \ar[r] & H^1(\G,\cF_{\bar\eta}^I) \ar[r] &
    H^1(G,\cF_{\bar\eta}) \ar[r] & H^1(I,\cF_{\bar\eta})^\G
    \ar[r] & 0
  }
\end{equation}
for a suitable vertical map $\phi$, where the top row is the
localisation sequence in $K'$-theory, so as to compare the kernels of
$\alpha$ and $\beta$. We recall (see \S3) that under the boundary map
$\partial$, the subspace $K_q^{(n)}U \subset K_qU\otimes\Q$ maps into
the subspace $\Kp{q-1}{n-d-1}Y\subset K'_{q-1}Y\otimes\Q$, and that
the Riemann-Roch transformation $\tau$ maps $\Kp{q-1}{n-d-1}Y$ to the
space of $\G$-invariants of the $\ell$-adic homology group
\begin{align*}
  H_{2d-i+1}(\Ybar,\Ql(d-n+1)) &= H^{i-2d-1}(\Ybar,Rf_s^!\Ql(n-d-1)) \\
  &\simeq H^{2d-i+1}(\Ybar,\Ql(d-n+1))^\vee
\end{align*}
(the isomorphism being given by Grothendieck-Verdier duality). In the
bottom row, we have
\[
H^1(I,\cF_{\bar\eta})=H^{i-1}(\Ubar,\Ql(n-1))_I
\simeq \left[ H^{2d-i+1}(\Ubar,\Ql(d-n+1))^I\right]^\vee
\]
by Poincar\'e duality. Finally we have the specialisation map
\[
sp\colon H^{2d-i+1}(\Ybar,\Ql) \to H^{2d-i+1}(\Ubar,\Ql)^I
\]
and we can therefore formulate the desired compatibility as:
\begin{prop}
  \label{prop:compat1}
  The following diagram is commutative up to sign:
  \[
  \xymat{
    K_q^{(n)}U \ar[r] \ar[dd]_{\AJl} & {}\Kp{q-1}{n-d-1}Y \ar[d]^\tau \\
    & H^{i-2d-1}(\Ybar,Rf_s^!\Ql(n-d-1))\cut{^\G} \ar[d]^{\simeq} \\
    H^1(G,H^{i-1}(\Ubar,\Ql(n))) \ar[d] 
    & {}\cut{\left[} H^{2d-i+1}(\Ybar,\Ql(d-n+1))^\vee \cut{\right]^\G} \\
    H^1(I,H^{i-1}(\Ubar,\Ql(n)))^\G 
    \cut{\ar[r]^{=} } \har[r]
    & {}\cut{\left[} (H^{2d-i+1}(\Ubar,\Ql(d-n+1))^\vee)_I \cut{\right]^\G}
    \ar[u]_{sp^\vee}
  }
  \]
\end{prop}

This will be reformulated in a more general setting in the next
section. First, we draw some consequences from it. We recall that the
monodromy-weight conjecture implies:

\begin{conj}[Local invariant cycle ``theorem'']
  \label{conj:ict}
  Suppose that $X$ is regular. Then for every $j$ the specialisation map
  \[
  sp\colon H^j(\Ybar,\Ql) \to H^j(\Ubar,\Ql)^I
  \]
  is a surjection.
\end{conj}

From \ref{prop:compat1} one then obtains immediately:

\begin{corol}[Jannsen]
\label{cor:uwe}
  Suppose that the Riemann-Roch transformation 
  \[
  \tau\colon \Kp{q-1}{d-n-1}Y\otimes\Ql \to
  H_{2d-2n+q+1}(\Ybar,\Ql(d-n+1))
  \]
  is injective, and that the local invariant cycle theorem
  \ref{conj:ict} holds for $(X,i-1)$. Then $\HMO^i(U,n)=\HMnr^i(U,n)$.
\end{corol}

The hypothesis that $\tau$ is injective would be a consequence of
Jannsen's generalisation of the Tate conjecture:

\begin{conj}[Jannsen {\cite[12.4(a)]{JMM}}]
  \label{conj:gentate}
  If $Y$ is proper over a finite field $k$, of dimension $d$, then
  Frobenius acts semisimply on the $\ell$-adic homology of $\Ybar$,
  and for every $q$ and $m$ the Riemann-Roch transformation is an isomorphism
  \[
  \tau\colon \Kp{q}{m}Y\otimes\Ql \isomarrow
  H_{q-2m}(\Ybar,\Ql(-m))^\G.
  \]
\end{conj}
As is shown in \cite[12.7]{JMM}, this is equivalent to standard
conjectures for $K$-theory of nonsingular varieties over finite fields:
\begin{conj}[Tate, Parshin]
  \label{conj:parshin}
  Let $Y$ be proper and smooth over a finite field $k$.
  \begin{itemize}
  \item The action of $\Gal(\bar k/k)$ on $H^*(\Ybar,\Ql)$ is semisimple.
  \item The cycle class map $CH^*(Y)\otimes\Ql \to
    H^{2*}(\Ybar,\Ql(*))^{\Gal(\bar k/k)}$ is an isomorphism.
  \item If $q>0$, then $K_qY\otimes\Q=0$.
  \end{itemize}
\end{conj}
(Jannsen's proof that \ref{conj:parshin} implies
\ref{conj:gentate} assumes resolution of singularities, but one can
remove this by appealing instead to De~Jong's alterations theorem
\cite{DJ}.)

We now analyze the proof in more detail to obtain Theorem
\ref{thm:main}. Granted Proposition \ref{prop:compat1}, It suffices to
prove the following two Propositions.

\begin{prop}
  \label{prop:mwcurves}
  Let $U= C_1 \times\dots\times C_d$ be a product of smooth proper curves.
  Then for all $j$, the monodromy-weight conjecture holds for
  $H^j(\Ubar)$.
\end{prop}

\begin{proof}
  The monodromy-weight conjecture is stable under products (by the
  K\"unneth formula and \cite[(1.6.9)]{DeW2}), so in particular it
  holds if $U$ is a product of curves (even for products
  of varieties of dimension at most $2$, by \cite{RZ}).
\end{proof}

\begin{prop}
  \label{prop:parshin-curves}
  Let $U=C_1 \times\dots\times C_d$ be a product of smooth proper
  curves. Then after replacing $F$ by a finite extension,
  $U$ admits a proper regular model $X/\fo$ for which:
  \begin{enumerate}
  \item $Y$ is a strict normal crossings divisor on $X$, and for every
    intersection $Z$ of components of $Y$, the $\G$-module $H^*(\bar
    Z,\Ql)$ is semisimple.
  \item the Riemann-Roch transformation on the homology of the special
    fibre
    \[
    \tau \colon H^\cM_{2m-j}(Y,m)\otimes\Ql \to H_{2m-j}(\Ybar,\Ql(m))^\G
    \]
    is an isomorphism for $m\le 1$.
  \end{enumerate}
\end{prop}

\begin{proof}

  We first need to construct a suitable regular model for $U$. After
  passing to a finite extension of $F$ we may assume that each factor
  $C_\mu$ has semistable reduction, and further has a semistable model
  $D_\mu$ whose special fibre is a reduced strict normal crossing
  divisor, whose components and singular points are all rational over
  the residue field. Let $X'=\prod D_\mu$. Then $X'$ is regular apart from
  singularities which are locally smooth over a product of double
  points; that is, locally isomorphic, for the \'etale topology, to
  \[
  \Spec \fo[x_1,y_1,\dots,y_r,y_r,z_1,\dots,z_s]
  /(x_1y_1-\pi_F,\dots,x_ry_r-\pi_F). 
  \]
  Take $X\to X'$ to be the resolution given in \cite[Lemme 5.5]{DeFM}. The
  special fibre $Y=\cup Y_\alpha$ is a normal crossings divisor in
  $X$. Write as usual 
  \begin{align*}
    Y_J& = \bigcap_{\alpha\in J} Y_\alpha &\text{for  
      $J\subset \{\alpha\}$}
    \\
    \Y{q}&=\coprod_{\#J=q+1} Y_J &\text{for $q\ge0$}
  \end{align*}
  Then the description of the desingularisation as an iterated blowup
  \cite[\S2]{SMM} shows that each $Y_J$ belongs to $\cC_k$, the
  smallest class of smooth and proper schemes over $k$ such that
  \begin{enumerate}
  \item $\cC_k$ contains all products of smooth proper geometrically
    connected curves;
  \item If $W$ is in $\cC_k$ and $P\to W$ is a projective bundle, then
    $P$ is in $\cC_k$;
  \item If $Z\subset W$ with $W$ and $Z$ both in $\cC_k$, then the
    blowup of $W$ along $Z$ is in $\cC_k$.
  \end{enumerate}
  If $W$ is in $\cC_k$ and $\dim W=d$, then the Chow motive of $W$ can
  be computed using the fomulae for the Chow motives of projective
  bundles and blowups, and it is a sum of Chow motives of the form
  $\otimes_{1\le j\le s}h^1(D_j)\otimes\mathbb{L}^{\otimes t}$ for
  curves $D_j$ and some $t\ge0$ with $s+t\le d$. From this it follows
  that the $\ell$-adic cohomology of $Y_J$ is semisimple.

  Together with the inclusion maps $Y_{J'}\subset Y_J$ for $J'\supset
  J'$, the $\Y{q}$ form a strict simplicial scheme
  \[
  \xymatrix{
    {}\Y{\bullet} =  \Bigl[\ \cdots \ 
      \ar@<1.5ex>[r] \ar@{}@<0.5ex>[r]|-{\vdots} \ar@<-1.5ex>[r]
    \Bigr. &
      {}\ \Y{2} \ \tplar & {}\  \Y{1}\ \dblar & {} \Bigl. \ \Y{0} \Bigr]
  }
  \]
  and the homology, both $\ell$-adic and motivic, of $Y$ is computed
  by a spectral sequence: 
  \begin{align}
    \label{eq:homss}
    {}^{\cM}E_1^{rs}&=  H^{\cM}_{2m-s}(\Y{-r}, m) \Rightarrow
    H^{\cM}_{2m-r-s}(Y,m) \notag
    \\
    {}^{\ell}E_1^{rs}&=  H_{2m-s}(\Yb{-r}, m) \Rightarrow
    H_{2m-r-s}(\Ybar,m) 
  \end{align}
  In the $\ell$-adic spectral sequence, since the $Y_J$ are smooth and
  proper we can rewrite the $E_1$ terms as
  \[
  {}^\ell E_1^{rs} = H^{2d+2r-2m+s}(\Yb{-r}, d+r-m)
  \]
  which is pure of weight $s$, and semisimple by (i). So
  the term $({}^\ell E_1^{rs})^\G$ vanishes unless $s=0$, and so we
  may conclude that, after passing to $\G$-invariants, the spectral
  sequence degenerates to an identity
  \[
  H_{2m-j}(\Ybar,\Ql(m))^\G = H_j \left[
    H_{2m}(\Yb{\bullet},\Ql(m))^\G\right].
  \]
  Consider now the motivic spectral sequence. Its $E_1$-terms may be
  computed as $K$-theory:
  \[
  {}^{\cM}E_1^{rs} = \HM^{2d+2r-2m+s}(\Y{-r}, d+r-m) =
  K_{-s}^{(d+r-m)}\Y{-r}.
  \]
  We can then apply the following trivial extension of \cite[Theorem 4]{SoK}.

  \begin{thm}[Soul\'e]
    \label{thm:soule}
    Let $Z$ be in $\cC_k$, of dimension $\le d$. Then
    \begin{enumerate}
    \item for every $a>0$ and every $b\ge d-1$, $K_a^{(b)}Z=0$; and
    \item for $m=0$, $1$ the cycle class map $CH_m(Z)\otimes\Ql \to
      H^{2(d-m)}(\bar Z,\Ql(d-m))$ is an isomorphism.
    \end{enumerate}
\end{thm}
  
  \begin{proof}
    As observed above, the Chow motive of $Z$ is a submotive of the
    motive of the product of $d$ curves, to which Soul\'e's result
    applies. 
  \end{proof}
  
  In the present case, since $\dim \Y{-r}=d+r$, part (i) gives
  ${}^{\cM}E_1^{rs} =0 $ for all $s\ne0$, provided $m\le1$. Therefore
  the spectral sequence also reduces to an identity
  \[
  H^{\cM}_{2m-j}(Y,m) = H_j \left[H^{\cM}_{2m}(\Y{\bullet},m)\right]
  = H_j\left[CH_m(\Y{\bullet})\otimes\Q\right].
  \]
  By (ii) we also have for every $m\le 1$ an isomorphism of
  homological complexes
  \[
  CH_m(\Y{\bullet})\otimes\Ql \to H_{2m}(\Yb{\bullet},\Ql(m))^\G
  \]
  (for $m<0$ both complexes are obviously zero). Therefore by
  comparing homology we get that $\tau$ is an isomorphism.
\end{proof}

\clearpage
\section{Homological setting}

In this section, $S=\Spec\fo$ is to be any Henselian trait (the
spectrum of a Henselian discrete valuation ring), with generic and
closed points $\eta$, $s$, of residue characteristic different from
$\ell$, and $f\colon X \to S$ any quasi-projective and flat morphism
of relative dimension $d$. Label the morphisms:
\[
\xymat{
  Y \har[r]^{g} \ar[d]_{f_s} & X \ar[d]^f & U \hhar[l]_{h}
  \ar[d]^{f_\eta}
  \\
  s \har[r]_{i} & S & \eta \hhar[l]^{j}
}
\]
We will replace $K$-theory by $K'$-theory and \'etale cohomology by
homology. We review some facts from \cite{SoOp}. Recall that when $U$ is
smooth, the $\gamma$-filtration $F_\gamma^\bullet$ on $K_qU$ satisfies
\[
(F_\gamma^nK_qU)\otimes\Q = \bigoplus_{m\ge n} K^{(m)}_qU.
\]
In general one has an increasing filtration $F_\bullet$ on
$K'U\otimes\Q$ (defined by embedding $U$ in a smooth scheme $Z$ and
taking a shift of the $\gamma$-filtration on $K^ZU=K'U$). There are
modified Adams operators $\phi^k$ on $K'$-theory and, if
$\Kp{q}{n}{U}\subset K_q'U\otimes\Q$ denotes the
$(\phi^k=k^m)$-eigenspace, then
\[
F_{-n}(K'_qU\otimes\Q) = \bigoplus_{m\ge n}\Kp{q}{m}U.
\]
When $U$ is smooth the isomorphism $K_qU\isomarrow K'_qU$ carries
$F_\gamma^n(K_qU\otimes\Q)$ to $F_{d-n}(K'_qU\otimes\Q)$ and therefore
induces  isomorphisms $K_q{(n)}U \isomarrow \Kp{q}{n-d}U$.

In \cite{Gi} there are defined $\ell$-adic Riemann-Roch transformations
\[
\tau\colon K_q'U \to H_{q-2m}(U,\Ql(-m))
\]
whose target is $\ell$-adic homology, defined as 
\[
H_{-j}(U,\Ql(-m))= H^j(U,Rf_\eta^!\Ql(m)).
\]
When $U$ is smooth, the Riemann-Roch theorem shows that for the Adams
eigenspaces there is a commutative diagram
\[
\xymatrix{
K_q^{(n)} \ar[r]^-{ch} \ar[dd]_{\simeq} &
H^{2n-q}(U,\Ql(n))\ar[d]^{\simeq}_{(P.D.)}
\\
& H^{2n-2d-q}(U,Rf_\eta^!\Ql(n-d))\eqar[d]
\\
\Kp{q}{n-d}U \ar[r]^-{\tau} & H_{q+2d-2n}(U,\Ql(d-n))
}
\]
where the isomorphism labelled $(P.D.)$ is the ``Poincar\'e duality''
isomorphism given by $Rf_\eta^!\Ql=\Ql(d)[2d]$.

All this applies equally to $Y$. In \'etale homology there is a
boundary map
\[
\partial_\ell\colon H_{-i}(U,\Ql(-m)) \to H_{-i+1}(Y,\Ql(-m+1))
\]
defined as the composite
\begin{align*}
  H_{-i}(U,\Ql(-m))=H^i(U,Rf_\eta^!\Ql(m))
  \mapright{\partial}
  &H^{i+1}_Y(X,Rf^!\Ql(m))
  \\
  =& H^{i+1}(Y,Rg^!Rf_s^!\Ql(m))
  \\
  =&H^{i+1}(Y,Rf_s^!Ri^!\Ql(m))
  \\
  =&H^{i-1}(Y,Rf_s^!\Ql(m-1))
  \\
  =& H_{-i+1}(Y,\Ql(-m+1))
\end{align*}
using the purity $Ri^!\Ql=\Ql(-1)[-2]$ on $S$. The boundary maps
$\partial_{\cM}$ and $\partial_\ell$ in
$K'$-theory and \'etale homology are compatible: the square
\begin{equation}
  \label{eq:compat-bdy}
  \xymatrix{
    {}\Kp qmU \ar[r]^-{\tau} \ar[d]^{\partial_{\cM}}
    & H_{q-2m}(U,\Ql(-m))\ar[d]^{\partial_\ell}
    \\
    {}\Kp{q-1}{m-1}Y \ar[r]^-{\tau} & H_{q-2m+1}(Y,\Ql(-m+1))
  }
\end{equation}
is commutative, cf.~\cite[end of \S8.1]{JMM}. (The strange numbering
of the homological boundary map comes from the equality of the
dimensions of $U$ and $Y$; by considering $U$ as having dimension
$(d+1)$ --- as for example is done in \cite{Lev} --- would lead to a
more natural numbering).

We have a Hochschild-Serre spectral sequence in homology:
\[
E_2^{ab}=H^a(G,H_{-b}(U,\Ql(\bullet))) \Rightarrow
H_{-a-b}(U,\Ql(\bullet))
\]
and therefore, if $\Fil^n$ is the abutment filtration, so that
\[
\Fil^1 H_*(U,\Ql(\bullet))=\ker\left[ H_*(U,\Ql(\bullet)) \to
  H_*(\Ubar,\Ql(\bullet)) \right]
\]
there is an edge homomorphism
\[
e_1\colon \Fil^1 H_j(U,\Ql(\bullet)) \to H^1(G,
H_{j+1}(U,\Ql(\bullet))).
\]
Let $(\Kp{q}{m}{U})^0=\tau^{-1}(\Fil^1 H_{q-2m}(U,\Ql(-m)))\subset
\Kp{q}{m}{U}$.  We can then state the homological generalisation of
\ref{prop:compat1}. Let 
\[
sp' \colon H_*(\Ubar,\Ql(\bullet))_I \to H_*(\Ybar,\Ql(\bullet))
\]
be the transpose, for Grothendieck-Verdier duality, of the
specialisation map
\[
sp \colon H^*_c(\Ybar,\Ql(\bullet)) \to H^*_c(\Ubar,\Ql(\bullet))^I.
\]
\begin{prop}
  \label{prop:compat2}
  The following diagram is commutative up to sign:
  \[
  \xymatrix{
  {}(\Kp{q}{m}{U})^0 \ar[d]^{\tau} \ar[r]^{\partial_{\cM}} &
  {}\Kp{q-1}{m-1}Y \ar[d]^{\tau}
  \\
  \Fil^1 H_{q-2m}(U,\Ql(-m)) \ar[d]^{e_1} 
  & H_{q-2m+1}(Y,\Ql(1-m)) \ar[d]
  \\
  H^1(G,H_{q-2m+1}(\Ubar,\Ql(-m))) \ar[d]
  & H_{q-2m+1}(\Ybar,\Ql(1-m))
  \\
   H^1(I,H_{q-2m+1}(\Ubar,\Ql(-m))) \har[r]
  & H_{q-2m+1}(\Ubar,\Ql(1-m))_I \ar[u]_{sp'}
  }
  \]
\end{prop}

The compatibility of boundary maps \eqref{eq:compat-bdy} means that we
can get rid of the $K'$-theory and express \ref{prop:compat2} as a
purely cohomological compatibility. We shall state and prove this in
the next section.

\section{$\ell$-adic compatibility}

Since the target space in the diagram is the homology
$H_{q-2m+1}(\Ybar,\Ql(1-m))$ of the geometric special fibre, we may
replace $S$ by its strict Henselisation. Then we can remove the
twists, and Proposition \ref{prop:compat2} will follow from
the commutativity of the following diagram, for any $r\in\Z$:
\[
\xymat{
  {}\Fil^1H^{r+1}(U,Rf_\eta^!\La) \ar[r]^{e_1} \har[d]
  & H^1(I,H^r(\Ubar,Rf_{\bar\eta}^!\La)) \ar[d]^{=} 
  \\
  H^{r+1}(U,Rf_\eta^!\La) \ar[d]^{\partial}
  & H^r(\Ubar, Rf_{\bar\eta}^!\La)_I(-1) \ar[d]^{=}
  \\
  H^{r+2}(Y,Rg^!Rf^!\La) \ar[d]^{=}
  & \left[ H^{-r}_c(\Ubar, \La^\vee)^I(1)\right]^\vee \ar[d]^{sp^\vee}
  \\
  H^r(Y,Rf_s^!\La(-1)) \ar[r]^{=} 
  & \left[ H_c^{-r}(Y,\La^\vee)(1)\right]^\vee
}
\]
We may push this down onto $S$, where it becomes the case 
$K=Rf_*Rf^!\Ql$, $L=Rf_!\Ql$ of the following statement.

\begin{prop}
\label{prop:compat3}
Let $S$ be a strictly Henselian trait, with generic and closed points
$\eta$, $s$, whose residue characteristic is different from $\ell$.
Let $K$, $L\in \cD_c^+(S,\Ql)$ together with a pairing $K\otimes L \to
\Ql(1)$, inducing a cohomological pairing
\[
\beta\colon H^2_s(S,K) \otimes H^0(s,L_s) \to H^2_s(S,\Ql(1))=\Ql
\]
Then the following diagram is commutative up to sign:
\[
\xymat{ 
  {}\Fil^1H^1(\eta,K_\eta) \ar[r]^{e_1} \har[d]
  & H^1(\eta, \cH^0(K_\eta)) \ar[d]^{=} 
  \\
  H^1(\eta,K_\eta) \ar[dd]^{\partial}
  & {}\cH^0(K_{\bar\eta})_I(-1) \ar[d]^{\beta}
  \\
  & {}\left[ \cH^0(L_{\bar\eta})^I \right]^\vee \ar[d]^{sp^\vee}
  \\
  H^2_s(S,K) \ar[r]^{\beta}
  & H^0(s,L_s)^{\vee}
}
\]
\end{prop}

\begin{proof}
  We can check this by pairing the whole diagram with $H^0(S,L)$, and
  are therefore reduced to the commutativity of the diagram:
  \[
  \xymat{
    {}\Fil^1H^1(\eta,K_\eta) \otimes H^0(S,L) \ar[rr]^{e_1\otimes id} \har[d]
    && H^1(\eta, \cH^0(K_\eta))\otimes H^0(S,L) \ar[d]
    \\
    H^1(\eta,K_\eta) \otimes H^0(S,L) \ar[d]^{\partial\otimes id}
    && {}\cH^0(K_{\bar\eta})_I \otimes\cH^0(L_{\bar\eta})^I (-1) \ar[d]_{\beta}
    \\
    H^2_s(S,K) \otimes H^0(S,L) \ar[rr]^{\beta}
    && H^2_s(S,\Ql(1))=\Ql 
  }
  \]
  To prove this we enlarge it to the enormous diagram below:
  \begin{landscape}\[
    \xymat{
      {}\Fil^1H^1(\eta,K_\eta) \otimes H^0(S,L)
      \har[rr] \ar[d]_{id\otimes j^*}
      && H^1(\eta,K_\eta) \otimes H^0(S,L) 
      \ar[ddd] \ar[rd]^{\partial\otimes i^*}
      \\
      {}\Fil^1H^1(\eta,K_\eta) \otimes H^0(\eta,L_\eta)
      \ar[d]_{e_1\otimes e_0} \ar[rdd]^{\cup}
      \ar@{}[rrd]|-{\mlb1}
      && \ar@{}[dr]|-{\mlb2}
      &  H^2(s,Ri^!K) \otimes H^0(s,i^*L) \ar[d]_{\cup}
      \\
      H^1(\eta,\cH^0(K_\eta))\otimes H^0(\eta,\cH^0(L_\eta))
      \ar[d]_{\cup} \ar@{}[dr]|-{\mlb3}
      &&& H^2(s,Ri^!(K\otimes L)) \ar[dd]_{\beta}
      \\
      H^1(\eta,\cH^0(K_\eta)\otimes \cH^0(L_\eta)) \ar[d]_{\cup}
      & {}\Fil^1H^1(\eta,K_\eta\otimes L_\eta)
      \har[r] \ar[ld]_{e_1} \ar[d]_{\beta} \ar@{}[rd]|-{\mlb5}
      & H^1(\eta,K_\eta\otimes L_\eta)
      \ar[ur]^{\partial} \ar[d]_{\beta} &
      \\
      H^1(\eta, \cH^0(K_\eta\otimes L_\eta))
      \ar[d]_{\beta} \ar@{}[r]|-{\mlb4}
      & {}\Fil^1H^1(\eta,A(1))
      \eqar[r] \eqar[dl] \ar@{}[dr]|-{\mlb7}
      & H^1(\eta,A(1)) \ar[r]^{\partial} \ar@{}[ru]|-{\mlb6}
      & H^2(s,Ri^!A(1)) \ar[dl]^{cl(s)^{-1}}_{\simeq}
      \\
      H^1(\eta,A(1)) \ar[rr]^{\simeq}_{\text{Kummer}} && A
    }
    \]
  \end{landscape}
  \noindent
  The commutativity of the various parts of this diagram are as follows:
  
  Parts (1), (4) and (5) obviously commute, and (6) commutes by
  functoriality.
  
  Part (2) commutes up to sign by \cite[0.1]{SH}, and part
  (3) commutes by \cite[0.4]{SH}. The remaining compatibility is (7),
  which is anti-commutative by \cite[``Cycle'', 2.1.3]{SGA4.5}.
\end{proof}

\vspace{6mm}\noindent
   Department of Pure Mathematics and Mathematical Statistics\\
   Centre for Mathematical Sciences\\
   Wilberforce Road\\
   Cambridge\enspace CB3 0WB\\
   \texttt{a.j.scholl@dpmms.cam.ac.uk}%


\begin{thebibliography}{99}
\bibitem{Be1} 
  A. A. Beilinson:
  \emph{Higher regulators and values of $L$-functions}.
  J. Soviet Math. \textbf{30} (1985), 2036--2070

\bibitem{BeH}
  --- :
  \emph{Height pairings between algebraic cycles}.
  Current trends in arithmetical algebraic geometry, ed. K. Ribet,
  1--24, Contemp. Math. \textbf{67}, Amer. Math. Soc., Providence, RI,
  1985.
 
\bibitem{Be2}
  --- :
  \emph{Notes on absolute Hodge cohomology}.
  Applications of algebraic $K$-theory to algebraic geometry and number
  theory, Parts I, II (Boulder, Colo., 1983), 35-68, Contemp. Math., 55,
  Amer. Math. Soc., Providence, RI, 1986.

\bibitem{BK}
  S. Bloch, K. Kato:
  \emph{$L$-functions and Tamagawa numbers of motives}.
  The Grothendieck Festschrift, vol. I, 333--400, Progress in
  Mathematics \textbf{86}, Birkh\"auser, 1990

\bibitem{DeFM}
  P. Deligne:
  \emph{Formes modulaires et repr\'esentations $l$-adiques}.
  S\'em. Bourbaki \textbf{11} (1968-1969), Exp. 355


\bibitem{DeW2}
  --- :
  \emph{La conjecture de Weil II}.
  Publ. math. IHES \textbf{52} (1980), 313--428

\bibitem{Gi}
  H. Gillet:
  \emph{Riemann-Roch theorems for higher algebraic $K$-theory}.
  Advances in Math. \textbf{40} (1981), 203--289

\bibitem{JC}
  U. Jannsen:
  \emph{Continuous \'etale cohomology}.
  Math. Annalen \textbf{280} (1988), 207--245

\bibitem{JMM}
  --- :
  \emph{Mixed motives and algebraic $K$-theory}.
  Lecture notes math. \textbf{1400}, Springer-Verlag, 1990

\bibitem{DJ}
  A. J. de~Jong:
  \emph{Smoothness, semi-stability and alterations}.
  Publ. Math. IHES \textbf{83} (1996), 51--93

\bibitem{Lev}
  M. Levine:
  \emph{Techniques of localization in the theory of algebraic cycles}.
  J. Alg. Geom. \textbf{10} (2001), 299-363

\bibitem{Nek}
  J. Nekov\'a\v{r}:
  \emph{Syntomic cohomology and $p$-adic regulators}.
  Preprint (1998), 55pp

\bibitem{Niz}
  W. Nizio{\l}:
  \emph{On the image of $p$-adic regulators}.
  Invent. Math. \textbf{127} (1997), 375--400

\bibitem{RSS}
  M. Rapoport, N. Schappacher, P. Schneider (eds.):
  \emph{Beilinson's conjectures on special values of $L$-functions}.
  Perspectives in Mathematics, 4. Academic Press, Inc., Boston, MA, 1988

\bibitem{RZ}
  M. Rapoport, T. Zink:
  \emph{\"Uber die lokale Zetafunktion von Shimuravariet\"aten.
    Monodromiefiltration und verschwindende Zyklen in ungleicher
    Charakteristik}. 
  Inventiones math. \textbf{68} (1982), 21--101

\bibitem{SMM}
  A. J. Scholl:
  \emph{Motives for modular forms}.
  Inventiones math. \textbf{100} (1990), 419--430

\bibitem{SH}
  --- :
  \emph{Height pairings and special values of $L$-functions}.
  In: Motives, Seattle 1991, ed.~U.~Jannsen, S.~Kleiman, J-P.~Serre. 
  Proc Symp. Pure Math \textbf{55} (1994), part 1, 571--598

\bibitem{SI}
  --- :
  \emph{Integral elements in $K$-theory and products of modular
    curves}.
  The arithmetic and geometry of algebraic cycles (Banff, AB, 1998),
  467-489, NATO Sci. Ser. C Math. Phys. Sci., \textbf{548}, 2000  

\bibitem{SoK}
  C. Soul\'e:
  \emph{Groupes de Chow et $K$-th\'eorie de vari\'et\'es sur un corps fini}
  Math. Annalen \textbf{268} (1984), 317--345

\bibitem{SoOp}
  --- :
  \emph{Op\'erations en $K$-th\'eorie alg\'ebrique}.
  Canad.~J.~Math. \textbf{37} (1985), 488--550

\bibitem{Tamme}
  G. Tamme:
  \emph{The theorem of Riemann-Roch}.
  In: Beilinson's conjectures on special values of $L$-functions,
  ed.~M.~Rapoport, N.~Schappacher, P.~Schneider (Academic Press, 1988),
  103--168

\bibitem{SGA4.5}
  SGA4$\frac12$: \emph{Cohomologie \'etale}. Lecture notes math. \textbf{569}

\end{thebibliography}
\end{document}